\documentclass{article}
\usepackage{amsmath, amssymb, amsthm}

\usepackage{geometry}
\usepackage{hyperref}
\usepackage[T1]{fontenc}

\usepackage{indentfirst}
\newtheorem{theorem}{Theorem}[section]

\newtheorem*{theorem*}{Theorem}
\newtheorem*{remark*}{Remark}
\newtheorem*{problem*}{Problem}
\newtheorem*{conjecture*}{Conjecture}
\newtheorem{lemma}[theorem]{Lemma}

\usepackage{geometry}
\usepackage{hyperref}
\usepackage[T1]{fontenc}
\usepackage{indentfirst}
\theoremstyle{remark}
\newtheorem{remark}{Remark}

\usepackage{graphicx} 

\title{Mean Values and Normal Order of a Successive-Iterate Ratio of Dedekind's Arithmetic Function}

\author{Aimin Guo$^{*}$}

\date{}

\begin{document}

\maketitle
\vspace{-2.0em}

\begin{center}
\small

\textit{School of Mathematics and Statistics}\\
\textit{Anhui Normal University}\\
\textit{Wuhu 241002, P.~R.~China}

\vspace{0.4em}

$^{*}$Corresponding author:\\
\underline{18339681864@163.com}

\end{center}

\begin{abstract}
We obtain asymptotic formulas for the three sums
\[
\sum_{n\le x}\frac{\psi(\psi(n))}{\psi(n)},\qquad
\sum_{n\le x}\frac{\psi(n)}{\psi(\psi(n))},\qquad
\sum_{n\le x}\log\frac{\psi(\psi(n))}{\psi(n)},
\]
where \(\psi\) denotes Dedekind's arithmetic function.
We also determine the normal order of the successive-iterate ratio;
more precisely,
\[
\frac{\psi(\psi(n))}{\psi(n)}
\sim
\frac{6e^\gamma}{\pi^2}\log_3 n
\]
for almost all positive integers \(n\).
\end{abstract}



\section{Introduction}

For each positive integer \(n\), let
\[
G_n=(\mathbb Z/n\mathbb Z)^2,
\]
and let \(C(n)\) denote the number of cyclic subgroups of \(G_n\) of
order \(n\). Since the number of elements of order \(n\) in \(G_n\) is
\[
n^2\prod_{p\mid n}\left(1-\frac1{p^2}\right),
\]
and each cyclic subgroup of order \(n\) has \(\phi(n)\) generators, we
obtain
\[
\begin{aligned}
C(n)
&=
\frac{n^2\displaystyle\prod_{p\mid n}(1-p^{-2})}{\phi(n)}\\
&=
n\prod_{p\mid n}\left(1+\frac1p\right)
=
\psi(n),
\end{aligned}
\]
where \(\psi\) denotes Dedekind's arithmetic function.

Thus \(\psi(n)\) has a natural algebraic interpretation as the number
of maximal-order cyclic subgroups of \(G_n\). This interpretation
suggests considering the iteration
\[
n,\qquad C(n),\qquad C(C(n)),\qquad\ldots,
\]
or, equivalently,
\[
n,\qquad \psi(n),\qquad \psi(\psi(n)),\qquad\ldots.
\]
Accordingly, we study the successive-iterate ratio
\[
R_\psi(n)
:=
\frac{C(C(n))}{C(n)}
=
\frac{\psi(\psi(n))}{\psi(n)}.
\]

The behavior of iterates of arithmetic functions has been studied from
both normal-order and mean-value points of view. Erd\H{o}s, Granville,
Pomerance and Spiro~\cite{EGPS} developed systematic methods for the
normal behavior of iterates of arithmetic functions, with particular
emphasis on successive iterates of Euler's function. In particular,
they determined the normal order of ratios of consecutive iterates of
\(\phi\).

Mean-value questions for the first two iterates of Euler's function
were subsequently studied by Warlimont~\cite{Warlimont}. He considered
the three sums
\[
\sum_{n\le x}\frac{\phi(n)}{\phi(\phi(n))},
\qquad
\sum_{n\le x}\frac{\phi(\phi(n))}{\phi(n)},
\qquad
\sum_{n\le x}\log\frac{\phi(n)}{\phi(\phi(n))},
\]
and obtained asymptotic formulas for each of them. These results
naturally suggest the corresponding mean-value problem for Dedekind's
function.

Compositions involving Dedekind's function have also been studied from
several other perspectives. S\'andor~\cite{Sandor2005} investigated
inequalities and related problems for compositions of
\(\phi,\psi\), and \(\sigma\), while S\'andor and
T\'oth~\cite{SandorToth} developed general extremal-order results for
various compositions of arithmetical functions. Related average-order
and density questions for compositions of classical arithmetic
functions were considered by De Koninck and Luca~\cite{DeKoninckLuca}.

The purpose of the present paper is to establish the corresponding
Dedekind-\(\psi\) analogue of Warlimont's mean-value results.
 We prove asymptotic formulas for
\[
\sum_{n\le x}R_\psi(n),\qquad
\sum_{n\le x}\frac1{R_\psi(n)},\qquad
\sum_{n\le x}\log R_\psi(n).
\]
The identity
\[
R_\psi(n)
=
\prod_{p\mid\psi(n)}
\left(1+\frac1p\right)
\]
shows that these mean values are governed by the distribution of the
prime divisors of \(\psi(n)\). Our arguments make use of estimates for
small and large prime divisors of \(\psi(n)\) established in
Guo et al.~\cite{Guo}, together with classical prime-product and sieve
estimates.

As a companion result, we also obtain the normal order
\[
R_\psi(n)
\sim
\frac{6e^\gamma}{\pi^2}\log_3 n
\]
for almost all positive integers \(n\). Thus the mean-value results and
the normal-order result provide complementary descriptions of the
growth of the subgroup-counting function \(C(n)\) under one further
iteration.

Throughout the paper, \(p,q,r\) denote primes, \(\gamma\) denotes
Euler's constant, and \(\log_k x\) denotes the \(k\)-fold iterate of
the natural logarithm.

We now state the main results of this paper.

\begin{theorem}\label{thm:mean-values}
As \(x\to\infty\), we have
\begin{align}
\sum_{n\le x}
\frac{\psi(\psi(n))}{\psi(n)}
&\sim
\frac{6e^\gamma}{\pi^2}\,x\log_3 x,
\label{eq:mean-psi}\\
\sum_{n\le x}
\frac{\psi(n)}{\psi(\psi(n))}
&\sim
\frac{\pi^2e^{-\gamma}}{6}\,
\frac{x}{\log_3 x},
\label{eq:mean-reciprocal-psi}\\
\sum_{n\le x}
\log\frac{\psi(\psi(n))}{\psi(n)}
&\sim
x\log_4 x.
\label{eq:log-mean-psi}
\end{align}
\end{theorem}

In addition to these mean-value asymptotics, we obtain the
corresponding normal-order result.

\begin{theorem}\label{thm:normal-order}
For almost all positive integers \(n\),
\[
\frac{\psi(\psi(n))}{\psi(n)}
\sim
\frac{6e^\gamma}{\pi^2}\log_3 n.
\]
\end{theorem}

In terms of the counting function \(C(n)\) introduced above, the two
theorems describe the growth of \(C(n)\) under iteration. In particular,
Theorem~\ref{thm:normal-order} gives
\[
C(C(n))
\sim
\frac{6e^\gamma}{\pi^2}C(n)\log_3 n
\]
for almost all \(n\). Thus, for a typical \(n\), one further iteration
of the subgroup-counting function increases its value by a factor
asymptotic to
\[
\frac{6e^\gamma}{\pi^2}\log_3 n.
\]

\section{\bf Preliminaries}
\bigskip
The proofs of our main results use four auxiliary results:
Mertens-type product estimates, a small-prime-power divisibility
property for \(\psi(n)\), an estimate for the contribution of large
prime divisors of \(\psi(n)\), and a Brun--Titchmarsh type bound for
reciprocal sums of primes in arithmetic progressions. We record them
below.

\begin{lemma}\label{sieve1-lemma}
As $x\to\infty$, we have
\begin{align}
    \prod_{p \leq x} \left(1 - \frac{1}{p}\right) 
    &= \frac{e^{-\gamma}}{\log x} \left(1 + O\left(\frac{1}{\log x}\right)\right),
    \label{3}
     \\
    \prod_{p \leq x} \left(1 + \frac{1}{p}\right) 
    &= \frac{6e^{\gamma}}{\pi^2} \log x \left(1 + O\left(\frac{1}{\log x}\right)\right).
    \label{4}
\end{align}
\end{lemma}

\begin{proof}
These standard Mertens-type estimates are recorded in
Guo et al.~\cite{Guo}.
\end{proof}

\begin{lemma}\label{div-lemma}
Suppose that \(x\) is sufficiently large. Then there exists
a constant \(c_1>0\) such that, for all positive integers \(n\le x\),
with \(O(x/(\log_3x)^2)\) exceptions, \(\psi(n)\) is divisible by every
prime power \(p^a\) satisfying
\[
p^a<c_1\frac{\log_2x}{\log_3x}.
\]
\end{lemma}

\begin{proof}
This is the \(\psi\)-part of
Guo et al.~\cite[Lemma~2.3]{Guo}.
\end{proof}

\begin{lemma}\label{hpsi-lemma}
Let \(x\) be sufficiently large, and define
\[
h_{\psi,x}(n)
:=
\sum_{\substack{p\mid\psi(n)\\p>\log_2x}}\frac1p.
\]
Then
\[
\sum_{n\le x}h_{\psi,x}(n)
\ll
\frac{x}{\log_3x}.
\]
Moreover, for almost all \(n\le x\),
\[
h_{\psi,x}(n)
<
\frac{\log_4x}{\log_3x}
=o(1).
\]
\end{lemma}

\begin{proof}
The first estimate is the \(\psi\)-part of
Guo et al.~\cite[Lemma~2.5]{Guo}. By Markov's inequality,
\[
\begin{aligned}
\#\left\{n\le x:
h_{\psi,x}(n)\ge
\frac{\log_4x}{\log_3x}\right\}
&\ll
\frac{x}{\log_4x}
=o(x).
\end{aligned}
\]
This proves the second assertion.
\end{proof}

\begin{lemma}\label{BT-lemma}
Let \(x\) be sufficiently large. Uniformly for primes \(q<x\),
\[
\sum_{\substack{r\le x\\r\equiv-1\pmod q}}
\frac1r
\ll
\frac{\log_2x}{q},
\]
where \(r\) runs over primes.
\end{lemma}

\begin{proof}
Let
\[
\pi(t;q,-1)
:=
\#\{r\le t:r\ \text{is prime},\ r\equiv-1\pmod q\}.
\]
If \(r\le2q\) and \(r\equiv-1\pmod q\), then
\(r=kq-1\) for some positive integer \(k\), and hence \(k=1\) or \(2\).
Thus
\[
\sum_{\substack{r\le2q\\r\equiv-1\pmod q}}\frac1r
\le
\frac1{q-1}+\frac1{2q-1}
\ll\frac1q.
\]
If \(x\le2q\), the desired estimate follows immediately. Hence we may
assume that \(x>2q\).

For \(t\ge2q\), since \((-1,q)=1\), the Brun--Titchmarsh inequality
\cite[Theorem~3.9]{MontgomeryVaughan} gives
\[
\pi(t;q,-1)
\ll
\frac{t}{\phi(q)\log(t/q)}
\ll
\frac{t}{q\log(t/q)},
\]
where we used \(\phi(q)=q-1\asymp q\).

By Abel summation,
\[
\begin{aligned}
\sum_{\substack{2q<r\le x\\r\equiv-1\pmod q}}\frac1r
&=
\frac{\pi(x;q,-1)}x
-\frac{\pi(2q;q,-1)}{2q}
+
\int_{2q}^x\frac{\pi(t;q,-1)}{t^2}\,dt\\
&\ll
\frac1q+
\frac1q\int_{2q}^x
\frac{dt}{t\log(t/q)}.
\end{aligned}
\]
With \(u=t/q\),
\[
\int_{2q}^x\frac{dt}{t\log(t/q)}
=
\int_2^{x/q}\frac{du}{u\log u}
=
\log\log(x/q)-\log\log2
\ll\log_2x.
\]
Combining this with the contribution from \(r\le2q\), we obtain
\[
\sum_{\substack{r\le x\\r\equiv-1\pmod q}}\frac1r
\ll
\frac{\log_2x}{q}.
\]
\end{proof}

\section{Proof of \eqref{eq:mean-psi} and
\eqref{eq:mean-reciprocal-psi} of Theorem~\ref{thm:mean-values}}

First, we observe that
\[
R_\psi(n)
=
\frac{\psi(\psi(n))}{\psi(n)}
=
\prod_{p\mid\psi(n)}
\left(1+\frac1p\right).
\]
Let $\mu$ denote the M\"obius function. Since $\mu^2(d)$ is the
characteristic function of the squarefree integers, for every positive
integer $m$,
\[
\prod_{p\mid m}\left(1+\frac1p\right)
=
\sum_{d\mid m}\frac{\mu^2(d)}{d}.
\]
Taking $m=\psi(n)$, we obtain
\[
R_\psi(n)
=
\sum_{d\mid\psi(n)}\frac{\mu^2(d)}{d}.
\]
Summing over $n\le x$, we obtain
\[
\sum_{n\le x}R_\psi(n)
=
\sum_{n\le x}
\sum_{d\mid\psi(n)}
\frac{\mu^2(d)}{d}.
\]
Since all sums are finite, we may interchange the order of summation.
Thus
\[
\sum_{n\le x}R_\psi(n)
=
\sum_{d\ge1}
\frac{\mu^2(d)}{d}
\sum_{\substack{n\le x\\ d\mid\psi(n)}}1.
\]
Define
\[
S_\psi(x,d)
:=
\#\{n\le x:d\mid\psi(n)\}.
\]
Then
\[{\sum_{n\le x}R_\psi(n)
=
\sum_{d\ge1}
\frac{\mu^2(d)}{d}\,S_\psi(x,d).}\]

Put
\[
y:=\log_2 x,
\]

and let \(P(d)\) denote the largest prime factor of \(d\), with
\(P(1):=1\). We split the above sum according to whether \(P(d)\le y\)
or \(P(d)>y\). Thus
\[
\sum_{n\le x}R_\psi(n)
=
\sum_{\substack{d\ge1\\P(d)\le y}}
\frac{\mu^2(d)}{d}\,S_\psi(x,d)
+
\sum_{\substack{d\ge1\\P(d)>y}}
\frac{\mu^2(d)}{d}\,S_\psi(x,d).
\]
Write
\[
A_1:=
\sum_{\substack{d\ge1\\P(d)\le y}}
\frac{\mu^2(d)}{d}\,S_\psi(x,d),
\]
and
\[
A_2:=
\sum_{\substack{d\ge1\\P(d)>y}}
\frac{\mu^2(d)}{d}\,S_\psi(x,d).
\]
Hence
\[
\sum_{n\le x}R_\psi(n)=A_1+A_2.
\]

For the contribution of the integers \(d\) with \(P(d)\le y\), we have
\[
A_1
\le
x\sum_{P(d)\le y}\frac{\mu^2(d)}d
=
x\prod_{p\le y}\left(1+\frac1p\right).
\]
By Lemma~\ref{sieve1-lemma},
\[
A_1
\le
\frac{6e^\gamma}{\pi^2}
x\log y
\left(1+O\left(\frac1{\log y}\right)\right).
\]
Since \(y=\log_2 x\), it follows that

\begin{equation}\label{A1}
A_1
\le
\frac{6e^\gamma}{\pi^2}
x\log_3 x
\left(1+O\left(\frac1{\log_3 x}\right)\right)
=
\left(\frac{6e^\gamma}{\pi^2}+o(1)\right)x\log_3 x.
\end{equation}

We now estimate
\[
A_2
:=
\sum_{\substack{d\ge1\\ P(d)>y}}
\frac{\mu^2(d)}{d}\,S_\psi(x,d).
\]
For each \(d\) occurring in \(A_2\), let
\[
q:=P(d).
\]
Since \(q\mid d\), the condition \(d\mid\psi(n)\) implies
\(q\mid\psi(n)\). Hence
\[
\{n\le x:d\mid\psi(n)\}
\subseteq
\{n\le x:q\mid\psi(n)\}.
\]
Therefore,
\[
S_\psi(x,d)
=
\#\{n\le x:d\mid\psi(n)\}
\le
\#\{n\le x:q\mid\psi(n)\}
=
S_\psi(x,q).
\]
Grouping the terms according to \(q=P(d)\), we obtain
\[
\begin{aligned}
A_2
&=
\sum_{q>y}
\sum_{P(d)=q}
\frac{\mu^2(d)}{d}\,S_\psi(x,d)\\
&\le
\sum_{q>y}
S_\psi(x,q)
\sum_{P(d)=q}
\frac{\mu^2(d)}{d}.
\end{aligned}
\]

Since $\mu^2(d)\neq0$ only when $d$ is squarefree, for $P(d)=q$
we may write uniquely
\[
d=qt,\qquad P(t)<q.
\]
In particular, $(q,t)=1$. Since $\mu^2$ is multiplicative and
$\mu^2(q)=1$ for the prime $q$, we have
\[
\mu^2(d)=\mu^2(qt)=\mu^2(q)\mu^2(t)=\mu^2(t).
\]
Therefore,
\[
\begin{aligned}
\sum_{P(d)=q}\frac{\mu^2(d)}{d}
&=
\sum_{P(t)<q}\frac{\mu^2(qt)}{qt}\\
&=
\frac1q\sum_{P(t)<q}\frac{\mu^2(t)}{t}\\
&
=\frac1q\prod_{p<q}\left(1+\frac1p\right)
\ll\frac{\log q}{q}.
\end{aligned}\]

where Lemma~\ref{sieve1-lemma} was used in the last step. Therefore,
\[
A_2
\ll
\sum_{q>y}\frac{\log q}{q}S_\psi(x,q).
\]

Recall that
\[
\psi(n)
=
n\prod_{r\mid n}\left(1+\frac1r\right).
\]
Thus, if a prime \(q\mid\psi(n)\), then either \(q^2\mid n\), or
there exists a prime divisor \(r\mid n\) such that
\[
q\mid r+1,
\]
that is,
\[
r\equiv-1\pmod q.
\]
Consequently,
\[
\begin{aligned}
S_\psi(x,q)
&=
\#\{n\le x:q\mid\psi(n)\}\\
&\le
\frac{x}{q^2}
+
\sum_{\substack{r\le x\\r\equiv-1\pmod q}}
\frac{x}{r}.
\end{aligned}
\]
Hence, by Lemma~\ref{BT-lemma},
\[
S_\psi(x,q)
\ll
\frac{x}{q^2}
+
\frac{x\log_2x}{q}
\ll
\frac{x\log_2x}{q},
\]
since \(q>y=\log_2x\).

It follows that
\[
A_2
\ll
x\log_2x
\sum_{q>y}\frac{\log q}{q^2}.
\]
To estimate the remaining prime sum, let
\[
\vartheta(t):=\sum_{q\le t}\log q.
\]
By Chebyshev's estimate,
\[
\vartheta(t)\ll t.
\]

For \(T>y\), partial summation gives
\[
\sum_{y<q\le T}\frac{\log q}{q^2}
=
\frac{\vartheta(T)}{T^2}
-\frac{\vartheta(y)}{y^2}
+
2\int_y^T\frac{\vartheta(t)}{t^3}\,dt.
\]
Since \(\vartheta(t)\ll t\), we have
\[
\frac{\vartheta(T)}{T^2}\to0
\qquad (T\to\infty).
\]
Hence, letting \(T\to\infty\),
\[
\sum_{q>y}\frac{\log q}{q^2}
\ll
\int_y^\infty\frac{dt}{t^2}
\ll
\frac1y.
\]
Therefore,
\[ A_2
\ll
x\log_2x
\sum_{q>y}\frac{\log q}{q^2}
\ll
\frac{x\log_2x}{y}.\]
   
Since \(y=\log_2x\), we obtain
\begin{equation}
A_2\ll x.\label{A2}
\end{equation}

Combining \eqref{A1} and \eqref{A2}, we obtain
\[
\begin{aligned}
A
:=\sum_{n\le x}R_\psi(n)
&=A_1+A_2\\
&\le
\frac{6e^\gamma}{\pi^2}
x\log_3x
\left(1+O\left(\frac1{\log_3x}\right)\right)
+O(x)\\
&=
\frac{6e^\gamma}{\pi^2}x\log_3x+O(x).
\end{aligned}
\]
Since \(O(x)=o(x\log_3x)\), it follows that
\begin{equation}\label{A}
{
A
\le
\left(\frac{6e^\gamma}{\pi^2}+o(1)\right)
x\log_3x.
}
\end{equation}

We next consider
\[
B
:=
\sum_{n\le x}\frac{1}{R_\psi(n)}
=
\sum_{n\le x}\frac{\psi(n)}{\psi(\psi(n))}.
\]
Put
\[
z:=c_1\frac{\log_2 x}{\log_3 x},
\]
where \(c_1\) is the constant appearing in Lemma~\ref{div-lemma}. By that lemma, there exists an exceptional set \(\mathcal E(x)\subseteq[1,x]\) with
\[
|\mathcal E(x)|
=
O\left(\frac{x}{(\log_3 x)^2}\right),
\]
such that, for every \(n\le x\) with \(n\notin\mathcal E(x)\),
all primes \(p<z\) divide \(\psi(n)\). Hence
\[
R_\psi(n)
=
\prod_{p\mid\psi(n)}
\left(1+\frac1p\right)
\ge
\prod_{p<z}\left(1+\frac1p\right).
\]
By Lemma~\ref{sieve1-lemma},
\[
\frac1{R_\psi(n)}
\le
\frac{\pi^2e^{-\gamma}}{6\log z}
\left(1+O\left(\frac1{\log z}\right)\right).
\]
Since
\[
\log z
=
\log_3x-\log_4x+O(1)
=
(1+o(1))\log_3x,
\]
we obtain
\[
\frac1{R_\psi(n)}
\le
\left(\frac{\pi^2e^{-\gamma}}6+o(1)\right)
\frac1{\log_3x}
\qquad
(n\notin\mathcal E(x)).
\]

On the other hand, \(R_\psi(n)\ge1\) for every \(n\), and hence
\(1/R_\psi(n)\le1\). Therefore,
\begin{equation}
\begin{aligned}\label{B}
B
&=
\sum_{\substack{n\le x\\n\notin\mathcal E(x)}}
\frac1{R_\psi(n)}
+
\sum_{\substack{n\le x\\n\in\mathcal E(x)}}
\frac1{R_\psi(n)}\\
&\le
\left(\frac{\pi^2e^{-\gamma}}6+o(1)\right)
\frac{x}{\log_3x}
+
O\left(\frac{x}{(\log_3x)^2}\right)\\
&=
\left(\frac{\pi^2e^{-\gamma}}6+o(1)\right)
\frac{x}{\log_3x}.
\end{aligned}
\end{equation}

It remains to derive the corresponding lower bounds. By the
Cauchy--Schwarz inequality,
\[
\left(\sum_{n\le x}a_nb_n\right)^2
\le
\left(\sum_{n\le x}a_n^2\right)
\left(\sum_{n\le x}b_n^2\right).
\]
Taking
\[
a_n=\sqrt{R_\psi(n)},
\qquad
b_n=\frac{1}{\sqrt{R_\psi(n)}},
\]
we have \(a_nb_n=1\), and hence
\[
\begin{aligned}
\left(\sum_{n\le x}1\right)^2
&\le
\left(\sum_{n\le x}R_\psi(n)\right)
\left(\sum_{n\le x}\frac{1}{R_\psi(n)}\right)\\
&=AB.
\end{aligned}
\]
Thus
\[
\lfloor x\rfloor^2\le AB.
\]
Since \(\lfloor x\rfloor=x+O(1)\), we have
\[
\lfloor x\rfloor^2=(1+o(1))x^2.
\]
Combining \eqref{A} and \eqref{B}, we obtain
\[
\begin{aligned}
AB
&\le
\left(\frac{6e^\gamma}{\pi^2}+o(1)\right)
x\log_3x
\left(\frac{\pi^2e^{-\gamma}}6+o(1)\right)
\frac{x}{\log_3x}\\
&=(1+o(1))x^2,
\end{aligned}
\]
since
\[
\frac{6e^\gamma}{\pi^2}
\cdot
\frac{\pi^2e^{-\gamma}}6
=1.
\]
Therefore,
\[
AB\sim x^2.
\]

We now derive the corresponding lower bounds for \(A\) and \(B\).
From \(AB\ge\lfloor x\rfloor^2\) and \eqref{B},
\[
\begin{aligned}
A
&\ge
\frac{\lfloor x\rfloor^2}{B}\\
&\ge
\frac{(1+o(1))x^2}
{\left(\frac{\pi^2e^{-\gamma}}6+o(1)\right)
x/\log_3x}\\
&=
\left(\frac{6e^\gamma}{\pi^2}+o(1)\right)
x\log_3x.
\end{aligned}
\]
Together with the upper bound \eqref{A}, this gives
\begin{equation}\label{Anormal}
{A\sim\frac{6e^\gamma}{\pi^2}x\log_3x.}
\end{equation}

Similarly, using \(AB\ge\lfloor x\rfloor^2\) and \eqref{A},
\[
\begin{aligned}
B
&\ge
\frac{\lfloor x\rfloor^2}{A}\\
&\ge
\frac{(1+o(1))x^2}
{\left(\frac{6e^\gamma}{\pi^2}+o(1)\right)
x\log_3x}\\
&=
\left(\frac{\pi^2e^{-\gamma}}6+o(1)\right)
\frac{x}{\log_3x}.
\end{aligned}
\]
Together with the upper bound \eqref{B}, we conclude that
\begin{equation}\label{Bnormal}
{B\sim\frac{\pi^2e^{-\gamma}}6\frac{x}{\log_3x}.}
\end{equation}

This completes the proof of \eqref{eq:mean-psi} and
\eqref{eq:mean-reciprocal-psi} in Theorem~\ref{thm:mean-values}.

\section{Proof of \eqref{eq:log-mean-psi} of Theorem~\ref{thm:mean-values}}

Put
\[
L_\psi(x)
:=
\sum_{n\le x}\log R_\psi(n).
\]
Since
\[
R_\psi(n)
=
\prod_{p\mid\psi(n)}
\left(1+\frac1p\right),
\]
we have
\[
\log R_\psi(n)
=
\sum_{p\mid\psi(n)}
\log\left(1+\frac1p\right).
\]

We first prove the upper bound. Put
\[
y:=\log_2x.
\]
Since \(\log(1+u)\le u\) for \(u>0\), we have
\[
\begin{aligned}
L_\psi(x)
&\le
\sum_{n\le x}
\sum_{\substack{p\mid\psi(n)\\ p\le y}}\frac1p
+
\sum_{n\le x}
\sum_{\substack{p\mid\psi(n)\\ p>y}}\frac1p\\
&\le
x\sum_{p\le y}\frac1p
+
\sum_{n\le x}h_{\psi,x}(n).
\end{aligned}
\]
Since \(y\to\infty\) as \(x\to\infty\), Mertens' formula gives
\[
\sum_{p\le y}\frac1p
=
\log_2y+O(1)
=
\log_4x+O(1).
\]
Moreover, by Lemma~\ref{hpsi-lemma},
\[
\sum_{n\le x}h_{\psi,x}(n)
\ll
\frac{x}{\log_3x}.
\]
Therefore,
\[
L_\psi(x)
\le
x\log_4x+O(x),
\]
and hence
\begin{equation}\label{L-upper}
L_\psi(x)
\le
x\log_4x+O(x).
\end{equation}

For the lower bound, put
\[
z:=c_1\frac{\log_2x}{\log_3x}.
\]
By Lemma~\ref{div-lemma}, apart from an exceptional set
\(\mathcal E(x)\) satisfying
\[
|\mathcal E(x)|
=
O\left(\frac{x}{(\log_3x)^2}\right),
\]
every prime \(p<z\) divides \(\psi(n)\). Thus, for
\(n\notin\mathcal E(x)\),
\[
R_\psi(n)
\ge
\prod_{p<z}\left(1+\frac1p\right).
\]
By Lemma~\ref{sieve1-lemma},
\[
\prod_{p<z}\left(1+\frac1p\right)
=
\frac{6e^\gamma}{\pi^2}\log z
\left(1+O\left(\frac1{\log z}\right)\right).
\]
Since
\[
\log z
=
\log_3x-\log_4x+O(1)
=
(1+o(1))\log_3x,
\]
taking logarithms gives
\[
\log R_\psi(n)
\ge
\log\log z+O(1)
=
\log_4x+O(1)
\qquad
(n\notin\mathcal E(x)).
\]

Since \(R_\psi(n)\ge1\) for every \(n\), we have
\(\log R_\psi(n)\ge0\). Therefore,
\[
\begin{aligned}
L_\psi(x)
&\ge
\sum_{\substack{n\le x\\n\notin\mathcal E(x)}}
\log R_\psi(n)\\
&\ge
\left(
x-O\left(\frac{x}{(\log_3x)^2}\right)
\right)
\left(\log_4x+O(1)\right)\\
&=
x\log_4x+O(x).
\end{aligned}
\]
Thus
\begin{equation}\label{L-lower}
L_\psi(x)
\ge
x\log_4x+O(x).
\end{equation}

Combining \eqref{L-upper} and \eqref{L-lower}, we obtain
\[
L_\psi(x)
=
x\log_4x+O(x).
\]
Since \(\log_4x\to\infty\) as \(x\to\infty\), it follows that
\[
L_\psi(x)
\sim
x\log_4x.
\]
Therefore,
\[
\sum_{n\le x}
\log\frac{\psi(\psi(n))}{\psi(n)}
\sim
x\log_4x.
\]

This completes the proof of \eqref{eq:log-mean-psi}.

\section{Proof of Theorem~\ref{thm:normal-order}}
Put
\[
R_\psi(n)
=
\frac{\psi(\psi(n))}{\psi(n)}
=
\prod_{p\mid\psi(n)}
\left(1+\frac1p\right).
\]

We first establish a lower bound. Let
\[
z:=c_1\frac{\log_2x}{\log_3x},
\]
where \(c_1\) is the constant appearing in
Lemma~\ref{div-lemma}. Apart from
\(O(x/(\log_3x)^2)=o(x)\) exceptions, every prime \(p<z\)
divides \(\psi(n)\). Hence
\[
R_\psi(n)
\ge
\prod_{p<z}\left(1+\frac1p\right).
\]
By Lemma~\ref{sieve1-lemma},
\[
\prod_{p<z}\left(1+\frac1p\right)
=
\frac{6e^\gamma}{\pi^2}\log z
\left(1+O\left(\frac1{\log z}\right)\right).
\]
Since
\[
\log z
=
(1+o(1))\log_3x,
\]
we obtain
\begin{equation}\label{min}
R_\psi(n)
\ge
\left(\frac{6e^\gamma}{\pi^2}-o(1)\right)
\log_3x
\end{equation}

for all but \(o(x)\) integers \(n\le x\).

For the upper bound, put \(y=\log_2x\). For every \(n\le x\),
\[
\begin{aligned}
R_\psi(n)
&=
\prod_{\substack{p\mid\psi(n)\\p\le y}}
\left(1+\frac1p\right)
\prod_{\substack{p\mid\psi(n)\\p>y}}
\left(1+\frac1p\right)\\
&\le
\prod_{p\le y}\left(1+\frac1p\right)
e^{h_{\psi,x}(n)},
\end{aligned}
\]
where we used \(1+u\le e^u\). By Lemma~\ref{sieve1-lemma},
\[
\prod_{p\le y}\left(1+\frac1p\right)
=
\left(\frac{6e^\gamma}{\pi^2}+o(1)\right)\log_3x,
\]
and hence
\[
R_\psi(n)
\le
\left(\frac{6e^\gamma}{\pi^2}+o(1)\right)
\log_3x\,e^{h_{\psi,x}(n)}.
\]

By Lemma~\ref{hpsi-lemma}, for almost all \(n\le x\),
\[
h_{\psi,x}(n)
<
\frac{\log_4x}{\log_3x}
=
o(1).
\]
Hence
\[
e^{h_{\psi,x}(n)}
=
1+o(1)
\]
for almost all \(n\le x\). Consequently,
\begin{equation}\label{max}
R_\psi(n)
\le
\left(\frac{6e^\gamma}{\pi^2}+o(1)\right)\log_3x
\end{equation}
for almost all \(n\le x\).

Combining \eqref{min} and \eqref{max}, and noting that the union of
the corresponding exceptional sets has cardinality \(o(x)\), we obtain
\[
R_\psi(n)
=
\left(\frac{6e^\gamma}{\pi^2}+o(1)\right)\log_3x
\]
for all but \(o(x)\) integers \(n\le x\).

Finally, for \(\sqrt{x}<n\le x\),
\[
\log_3n
=
\log_3x+O\left(\frac1{\log_2x}\right)
\sim\log_3x
\]
uniformly, while the omitted integers \(n\le\sqrt{x}\) are \(o(x)\)
in number. Hence
\[
\frac{\psi(\psi(n))}{\psi(n)}
\sim
\frac{6e^\gamma}{\pi^2}\log_3n
\]
for almost all positive integers \(n\). This completes the proof of Theorem~\ref{thm:normal-order}.

\section{\bf Concluding Remark}
\medskip

\begin{remark}
The preceding results naturally raise the question of the behavior of
higher iterates of the Dedekind function. For each fixed integer
\(k\ge2\), let
\[
R_{\psi,k}(n)
:=
\frac{\psi^{(k+1)}(n)}{\psi^{(k)}(n)},
\]
where \(\psi^{(k)}\) denotes the \(k\)-fold iterate of \(\psi\).
It would be interesting to determine the average and normal orders of
\(R_{\psi,k}(n)\). In particular, does the successive-iterate ratio
continue to admit a normal order for every fixed \(k\ge2\), and how
does this order depend on \(k\)?
\end{remark}

\end{document}